\documentclass[11pt]{article}

\usepackage[T1]{fontenc}
\usepackage[utf8]{inputenc}
\usepackage{lmodern}
\usepackage{amsmath,amssymb,amsthm,amscd,mathtools,mathrsfs}
\usepackage{microtype}
\usepackage[margin=1.08in]{geometry}
\usepackage{enumitem}
\usepackage{aliascnt}
\usepackage[hidelinks]{hyperref}

\theoremstyle{definition}
\newtheorem{theorem}{Theorem}[section]
\newaliascnt{proposition}{theorem}
\newtheorem{proposition}[proposition]{Proposition}
\aliascntresetthe{proposition}
\newaliascnt{lemma}{theorem}
\newtheorem{lemma}[lemma]{Lemma}
\aliascntresetthe{lemma}
\newaliascnt{corollary}{theorem}
\newtheorem{corollary}[corollary]{Corollary}
\aliascntresetthe{corollary}
\newaliascnt{remark}{theorem}

\aliascntresetthe{remark}

\setlist[enumerate]{itemsep=2pt,topsep=5pt,parsep=0pt,partopsep=0pt}
\title{Lazard's Realization Problem for \(N\)-Series:\\ Bar Obstructions and Finite-Quotient Descent}
\author{Chao Ma\thanks{Independent Researcher, London, United Kingdom.
Email: \href{mailto:raymond.ma@me.com}{\texttt{raymond.ma@me.com}}.
ORCID: \href{https://orcid.org/0009-0004-2456-9098}
{0009-0004-2456-9098}.}}
\date{}

\hypersetup{
  pdftitle={Lazard's Realization Problem for N-Series: Bar Obstructions and Finite-Quotient Descent},
  pdfauthor={Chao Ma},
  pdfcreator={Chao Ma},
  pdfsubject={Lazard's question on recovering a prescribed N-series from a multiplicative filtration of the integral group ring, the bar-cokernel obstruction, and failure of descent at the fourth dimension subgroup},
  pdfkeywords={N-series, augmentation filtration, dimension subgroup, integral group ring, Lazard, normalized bar complex, cokernel obstruction, Losey expression, finite-quotient descent, fourth dimension subgroup, Tahara, Hartl-Mikhailov-Passi, group ring filtration}
}

\begin{document}
\setcounter{section}{-1}
\maketitle

\begin{abstract}
Lazard asked which \(N\)-series of a group arise from multiplicative
filtrations of the integral group ring.  Given an \(N\)-series
\(H_\bullet\) of a group \(G\), let \(A_\bullet\) be the filtration of the
augmentation ideal spanned by products of elements \(x-1\), \(x\in H_a\),
weighted by \(a\).  Every realizing filtration contains \(A_\bullet\), so
\(H_\bullet\) is realizable exactly when the maps
\(H_n/H_{n+1}\to A_n/A_{n+1}\), \(h\mapsto h-1\), are injective.  We show
that each kernel is the cokernel of an explicit map of normalized bar
groups; for a finite group this makes realizability
decidable by integer linear algebra.  For a profinite group with a cofinal
tower of finite quotients and an \(N\)-series induced from it, suppose \(h\in H_n\) and \(h-1\in A_{n+1}\) in
every finite quotient.  Then \(h-1\in A_{n+1}\) in the discrete group ring
exactly when the widths of its finite-level Losey expressions are bounded.
This descent can fail.  Using Tahara's class-three \(2\)-groups and the
Hartl--Mikhailov--Passi description of the fourth dimension subgroup, we
construct a countable product \(P\) of finite \(2\)-groups.  It contains an
element of \(\gamma_3(P)\) that lies in the fourth dimension subgroup of
every finite subproduct of \(P\) but not in \(D_4(P)\).
\end{abstract}

\section{Introduction}\label{introduction-and-main-results}

An \(N\)-series of a group \(G\), in the sense of Lazard~\cite{Lazard1954},
is a descending chain
\[
G=H_1\ge H_2\ge\cdots,\qquad
[H_a,H_b]\le H_{a+b}.
\tag{0.1}
\]
Lazard's theorem \cite[Ch.~I, Th\'eor\`eme~3.2]{Lazard1954} attaches to a
filtered associative algebra \(R\) and a homomorphism \(f:G\to R^\times\) with \(v(f(x)-1)\ge1\) the \(N\)-series
\(H_i=\{x:v(f(x)-1)\ge i\}\).  In Chapter~I, \S8 he asks which
\(N\)-series of a given group arise in this way.  For \(R=\mathbf Z[G]\) this
is the question studied by Losey \cite[\S1]{Losey1974}, who showed that the
answer is always positive when \(G\) is abelian
\cite[Theorem~2]{Losey1974}; for the shifted congruence series of the
special Nottingham jet groups the answer is also positive
\cite{Ma2026NottinghamAugmentation}.  The weighted
augmentation filtration \(A_\bullet\) of Section~\ref{weighted-augmentation-filtration}
gives maps
\[
\phi_n:H_n/H_{n+1}\longrightarrow A_n/A_{n+1},
\qquad
hH_{n+1}\longmapsto(h-1)+A_{n+1}.
\tag{0.2}
\]
For the lower central series \(H_n=\gamma_n\), injectivity of these maps is
the dimension subgroup problem.  Sj\"ogren proved that \(D_n(G)/\gamma_n(G)\) is
annihilated by an explicit integer depending only on \(n\)
\cite{Sjogren1979}; Bartholdi and Mikhailov showed that every finite abelian
group occurs in such a quotient \cite{BartholdiMikhailov2020}, and the
dimension quotients are now described in all degrees as boundary limits over
the category of presentations \cite{Mikhailov2026}.  Lazard's question concerns an arbitrary prescribed \(N\)-series.

Lazard's realization problem asks when these maps are injective, or
equivalently when \(G\cap(1+A_n)=H_n\) for every \(n\)
(Proposition~\ref{proposition-1.1-equivalence-with-lazards-dimension-equality}
and Corollary~\ref{corollary-1.2-canonical-minimality}).  One obstruction is
intrinsic to the integral group ring.  A second appears only on passing from
finite quotients back to \(\mathbf Z[G]\): membership may hold at every
finite level without a single discrete expression witnessing it.  The main
results are the following.

\begin{enumerate}
\def\labelenumi{\arabic{enumi}.}
\item
  In every degree, \(\ker\phi_n\) is the cokernel of an explicit
  relative-bar map (Theorem~\ref{theorem-3.2-relative-bar-realization}).
\item
  For a finite separated \(N\)-series, the vanishing of these kernels is
  decided by integer linear algebra
  (Theorem~\ref{theorem-4.1-finite-criterion}).
\item
  For an \(N\)-series induced from a countable cofinal tower of finite
  quotients, discrete descent holds
  if and only if the finite-level widths of Losey expressions are bounded
  (Theorem~\ref{theorem-6.1-discrete-descent-bounded-width}).
\item
  Discrete descent fails in general
  (Theorem~\ref{theorem-8.3-failure-of-discrete-descent}): a countable
  product \(P\) of four-generated class-three finite \(2\)-groups contains an
  element
  \[
  h\in D_4^{\rm fin}(P)\setminus D_4(P).
  \tag{0.3}
  \]
\end{enumerate}

\section{The canonical filtration and the bar obstruction}\label{definitions}

\subsection{Weighted augmentation
filtration}\label{weighted-augmentation-filtration}

Let \(\mathbf Z[G]\) be the integral group ring and
\(\Delta(G)=\ker(\mathbf Z[G]\to\mathbf Z)\) its augmentation ideal. For
\(m\ge1\), define \(A_m=A_m(G,H_\bullet)\) to be the additive subgroup
generated by all products
\[
(x_1-1)\cdots(x_r-1),
\qquad
x_\nu\in H_{a_\nu},
\qquad
\sum_{\nu=1}^r a_\nu\ge m.
\tag{1.1}
\]
The \(N\)-series identities imply
\[
A_aA_b\subseteq A_{a+b},
\qquad
H_m-1\subseteq A_m.
\tag{1.2}
\]
Set
\[
D_m^{\rm alg}(G,H_\bullet)
:=
G\cap(1+A_m).
\tag{1.3}
\]
The map in (0.2) is a homomorphism, since for \(h,k\in H_n\),
\[
hk-1=(h-1)+(k-1)+(h-1)(k-1),
\]
and the last term lies in \(A_{2n}\subseteq A_{n+1}\). Replacing \(h\)
by an element of the same coset modulo \(H_{n+1}\) changes \(h-1\) by an
element of \(A_{n+1}\), so the induced map on \(H_n/H_{n+1}\) is also
well defined.

The subgroup \(A_m\) is a two-sided ideal of \(\mathbf Z[G]\). If
\(y=(x_1-1)\cdots(x_r-1)\) has total weight at least \(m\) and
\(g\in G=H_1\), then
\[
gy=y+(g-1)y,
\qquad
yg=y+y(g-1).
\tag{1.3a}
\]
Both added summands have total weight at least \(m+1\); since \(A_m\) is
additively generated by such products, it is an ideal.

By definition of \(\phi_n\),
\[
\ker\phi_n
=
\frac{H_n\cap D_{n+1}^{\rm alg}}{H_{n+1}},
\tag{1.4}
\]
that is, \(hH_{n+1}\in\ker\phi_n\) if and only if \(h-1\in A_{n+1}\).
We say that \(A_\bullet\) realizes \(H_\bullet\) when every \(\phi_n\) is
injective.  For the lower central series we write
\[
D_m(G):=G\cap\bigl(1+\Delta(G)^m\bigr)
\tag{1.4$'$}
\]
for the ordinary integral dimension subgroup.

\begin{proposition}[Equivalence with Lazard's dimension equality]
\label{proposition-1.1-equivalence-with-lazards-dimension-equality}

One has
\[
D_n^{\rm alg}(G,H_\bullet)=H_n\text{ for every }n
\iff
\phi_n\text{ is injective for every }n.
\tag{1.4a}
\]
\end{proposition}

\begin{proof}
If \(D_{n+1}^{\rm alg}=H_{n+1}\), formula (1.4) gives \(\ker\phi_n=0\).

For the converse, suppose every \(\phi_s\) is injective and
\(x\in D_n^{\rm alg}\). If
\(x\notin H_n\), choose \(s<n\) with
\[
x\in H_s\setminus H_{s+1}.
\]
Since \(A_n\subseteq A_{s+1}\), the nonzero class \(xH_{s+1}\) lies in
\(\ker\phi_s\), a contradiction. Hence \(D_n^{\rm alg}\subseteq H_n\). The
reverse inclusion follows from \(H_n-1\subseteq A_n\).
\end{proof}

\begin{corollary}[Canonical minimality]
\label{corollary-1.2-canonical-minimality}

Let \(I_\bullet\) be any descending multiplicative augmentation
filtration by two-sided ideals such that
\[
I_aI_b\subseteq I_{a+b},
\qquad
H_a-1\subseteq I_a.
\tag{1.4b}
\]
Then \(A_n\subseteq I_n\) for every \(n\). Hence there exists
such a filtration satisfying
\[
G\cap(1+I_n)=H_n\quad\text{for all }n
\tag{1.4c}
\]
if and only if the canonical filtration itself satisfies
\[
G\cap(1+A_n)=H_n\quad\text{for all }n.
\tag{1.4d}
\]
\end{corollary}

\begin{proof}
Every generator of \(A_n\) is a product of terms
\(x_\nu-1\in I_{a_\nu}\) of total weight at least \(n\), so (1.4b) gives
\(A_n\subseteq I_n\). If (1.4c) holds, then
\[
H_n\subseteq G\cap(1+A_n)
\subseteq G\cap(1+I_n)=H_n.
\tag{1.4e}
\]
The converse follows by taking \(I_\bullet=A_\bullet\).
\end{proof}

\subsection{The central extension and its
transgression}\label{quotient-edge-and-transgression}

Fix \(n\ge1\) and put
\[
E_n=G/H_{n+1},\qquad
Q_n=G/H_n,\qquad
C_n=H_n/H_{n+1}.
\tag{1.5}
\]
Then
\[
1\longrightarrow C_n\longrightarrow E_n
\xrightarrow{\pi_n}Q_n\longrightarrow1
\tag{1.6}
\]
is a central extension. Its integral homology five-term exact sequence
contains
\[
H_2(E_n;\mathbf Z)\longrightarrow H_2(Q_n;\mathbf Z)
\xrightarrow{\kappa_n}C_n
\longrightarrow H_1(E_n;\mathbf Z)
\longrightarrow H_1(Q_n;\mathbf Z)\longrightarrow0.
\tag{1.7}
\]
The map \(\kappa_n\) is the central transgression; its sign is fixed by
the factor-set formula (1.23).

\subsection{Normalized bar complex}\label{normalized-bar-complex}

For a group \(Q\), let \(\overline C_r(Q)\) be the free abelian group on
symbols
\[
[q_1|\cdots|q_r],\qquad q_i\ne1.
\tag{1.8}
\]
Terms containing an identity are interpreted as zero. With trivial
integral coefficients the boundary is
\[
\begin{aligned}
\partial[q_1|\cdots|q_r]
={}&[q_2|\cdots|q_r]\\
&+\sum_{i=1}^{r-1}(-1)^i
[q_1|\cdots|q_iq_{i+1}|\cdots|q_r]\\
&+(-1)^r[q_1|\cdots|q_{r-1}].
\end{aligned}
\tag{1.9}
\]
Write
\[
Z_2(Q)=\ker(\partial:\overline C_2(Q)\to\overline C_1(Q)),
\qquad
B_2(Q)=\operatorname{im}(\partial:\overline C_3(Q)\to\overline C_2(Q)).
\tag{1.10}
\]
Then \(H_2(Q;\mathbf Z)=Z_2(Q)/B_2(Q)\).

\subsection{The factor-set bar obstruction}\label{factor-set-bar-obstruction}

By Lemma~\ref{lemma-2.2-kernel-transgression} below, which applies
Losey's theorem after abelianization,
\[
\ker\phi_n\subseteq\operatorname{im}\kappa_n.
\tag{1.11}
\]
Choose a normalized set-theoretic section
\[
s_n:Q_n\longrightarrow E_n.
\tag{1.12}
\]
Its central factor set is
\[
\alpha_n(q,q')
:=s_n(q)s_n(q')s_n(qq')^{-1}\in C_n.
\tag{1.13}
\]
Define a homomorphism
\[
\vartheta_{n,s}:\overline C_2(Q_n)
\longrightarrow A_n/A_{n+1}
\tag{1.14}
\]
on normalized bar generators by
\[
\vartheta_{n,s}([q\mid q'])
=
(\widetilde\alpha_n(q,q')-1)+A_{n+1},
\tag{1.15}
\]
where \(\widetilde\alpha_n(q,q')\in H_n\) represents
\(\alpha_n(q,q')\in H_n/H_{n+1}\). The value is independent of the
representative: replacing \(\widetilde\alpha_n(q,q')\) by
\(\widetilde\alpha_n(q,q')h\), with \(h\in H_{n+1}\), changes the
representative of (1.15) by
\[
\widetilde\alpha_n(q,q')(h-1)\in A_{n+1}.
\tag{1.16}
\]
The factor set satisfies
\[
\alpha_n(q,q')\alpha_n(qq',q'')
=
\alpha_n(q',q'')\alpha_n(q,q'q'').
\tag{1.17}
\]
For a normalized three-chain generator,
\[
\partial[q\mid q'\mid q'']
=
[q'\mid q'']-[qq'\mid q'']
+[q\mid q'q'']-[q\mid q'].
\tag{1.18}
\]
Since the map \(C_n\to A_n/A_{n+1}\),
\(c\mapsto(\widetilde c-1)+A_{n+1}\), is \(\phi_n\), applying it to
(1.17) gives
\[
\begin{aligned}
\vartheta_{n,s}\partial[q\mid q'\mid q'']
={}&\phi_n(\alpha_n(q',q''))
-\phi_n(\alpha_n(qq',q''))\\
&+\phi_n(\alpha_n(q,q'q''))
-\phi_n(\alpha_n(q,q'))
=0.
\end{aligned}
\tag{1.19}
\]
Therefore \(\vartheta_{n,s}\) vanishes on \(B_2(Q_n)\), and its restriction
to cycles induces
\[
\overline\vartheta_n:H_2(Q_n;\mathbf Z)
\longrightarrow A_n/A_{n+1}.
\tag{1.20}
\]
To verify independence of the section, write
\(s'_n(q)=t(q)s_n(q)\), where \(t(q)\in C_n\) and
\(t(1)=1\). Centrality of \(C_n\) gives
\[
\alpha'_n(q,q')
=t(q)t(q')t(qq')^{-1}\alpha_n(q,q').
\tag{1.21}
\]
If \(\lambda([q])=\phi_n(t(q))\), then on normalized two-chains
\[
\vartheta_{n,s'}-\vartheta_{n,s}
=\lambda\partial_2,
\qquad
(\vartheta_{n,s'}-\vartheta_{n,s})(z)=0
\quad(z\in Z_2(Q_n)).
\tag{1.22}
\]
Let \(z=\sum m_{q,q'}[q\mid q']\) be a normalized two-cycle and
lift each two-simplex through \(s_n\). Its failure to close in \(E_n\) is
the kernel element \(\alpha_n(q,q')\); because \(z\) is a cycle, the
lifted one-dimensional faces cancel. With the chosen sign convention, the
transgression is
\[
\kappa_n([z])
=\prod_{q,q'}\alpha_n(q,q')^{m_{q,q'}}\in C_n.
\tag{1.23}
\]
Applying \(\phi_n\) to (1.23) and using (1.15) yields
\[
\overline\vartheta_n=\phi_n\kappa_n.
\tag{1.24}
\]
Define
\[
F_n:=\ker\overline\vartheta_n
\subseteq H_2(Q_n;\mathbf Z).
\tag{1.25}
\]
Let
\[
q_n:Z_2(Q_n)\twoheadrightarrow H_2(Q_n;\mathbf Z)
\tag{1.26}
\]
be the quotient map and put
\[
D_n^{\rm bar}:=q_n^{-1}(F_n)
=\ker\bigl(\vartheta_{n,s}|_{Z_2(Q_n)}\bigr).
\tag{1.27}
\]

\subsection{Finite quotient
saturation}\label{finite-quotient-saturation}

Let
\[
G=\varprojlim_{j\ge1}G_j
\tag{1.28}
\]
be a separated countably based profinite group with a fixed
descending cofinal tower of finite quotients and surjective transition
maps. Assume that the \(N\)-series is closed and induced from the finite
quotients, so that
\[
H_a=\varprojlim_jH_{a,j},
\qquad
H_{a,j}=q_j(H_a),
\tag{1.28a}
\]
with the transition maps carrying \(H_{a,j+1}\) onto \(H_{a,j}\). Let
\(A_{m,j}\subseteq\mathbf Z[G_j]\) be defined by (1.1). Every weighted
generator of \(A_{m,j}\) lifts factor by factor through
\(H_{a,j+1}\twoheadrightarrow H_{a,j}\), so the transition map of group
rings carries \(A_{m,j+1}\) onto \(A_{m,j}\).

The canonical map is
\[
\delta:\mathbf Z[G]\longrightarrow
\varprojlim_j\mathbf Z[G_j].
\tag{1.29}
\]
Define
\[
A_m^{\rm fin}
:=
\delta^{-1}\!\left(\varprojlim_jA_{m,j}\right)
\tag{1.30}
\]
and
\[
D_m^{\rm fin}
:=
G\cap(1+A_m^{\rm fin})
=
\bigcap_jq_j^{-1}\!\left(G_j\cap(1+A_{m,j})\right).
\tag{1.31}
\]

\section{Preliminaries}\label{preliminaries}

We use the following results of Losey, Hartl--Mikhailov--Passi, and Tahara.

\begin{theorem}[Losey {\cite[Theorem~2]{Losey1974}}]
\label{theorem-2.1-loseys-abelian-injectivity-theorem}

If \(V\) is abelian and \(V=V_1\ge V_2\ge\cdots\) is a descending
subgroup series, the canonical filtration satisfies
\[
V\cap(1+A_n(V,V_\bullet))=V_n
\quad\text{for every }n.
\tag{2.1}
\]
The induced maps
\[
\phi_n^V:V_n/V_{n+1}\longrightarrow
A_n(V,V_\bullet)/A_{n+1}(V,V_\bullet)
\tag{2.1a}
\]
are injective in every degree.
\end{theorem}

\begin{lemma}\label{lemma-2.2-kernel-transgression}

For an arbitrary prescribed \(N\)-series,
\[
\ker\phi_n\subseteq\operatorname{im}\kappa_n.
\tag{2.2}
\]
Moreover, with \(F_n\) as in (1.25),
\[
\kappa_n(F_n)=\ker\phi_n.
\tag{2.2a}
\]
\end{lemma}

\begin{proof}
Let \(h\in H_n\) represent a class in \(\ker\phi_n\). Abelianize \(G\),
and give \(G_{\rm ab}\) the image series
\(\overline H_i=H_i[G,G]/[G,G]\). Functoriality of (1.1) sends
\(h-1\in A_{n+1}(G,H_\bullet)\) into
\(\bar h-1\in A_{n+1}(G_{\rm ab},\overline H_\bullet)\). Theorem~\ref{theorem-2.1-loseys-abelian-injectivity-theorem}
gives \(\bar h\in\overline H_{n+1}\), hence
\[
h\in [G,G]H_{n+1}.
\tag{2.2b}
\]
In the central extension (1.6), this says that the class of \(h\) lies
in the kernel of
\[
C_n\longrightarrow H_1(E_n;\mathbf Z).
\]
Exactness of (1.7) identifies that kernel with
\(\operatorname{im}\kappa_n\), proving (2.2). If \(x\in F_n\),
then (1.24) gives
\[
\phi_n\kappa_n(x)=\overline\vartheta_n(x)=0,
\]
so \(\kappa_n(F_n)\subseteq\ker\phi_n\). For the reverse inclusion, let
\(y\in\ker\phi_n\); then (2.2) gives
\(x\in H_2(Q_n;\mathbf Z)\) with \(\kappa_n(x)=y\). Equation (1.24)
then gives \(\overline\vartheta_n(x)=0\), hence \(x\in F_n\). This
proves (2.2a).
\end{proof}

\begin{theorem}[Hartl--Mikhailov--Passi \cite{HMP}]
\label{theorem-2.3-hmp-map-degree-four}

Let \(X\) be nilpotent of class at most three and put
\[
\bar X=X/\gamma _3(X),\qquad
\mathsf A=X_{\rm ab}=\bar X_{\rm ab},\qquad
B=\bar X'=\gamma _2(X)/\gamma _3(X).
\tag{2.3}
\]
The quotient \(\bar X\) has class at most two. Write
\[
c_2:\Lambda ^2\mathsf A\twoheadrightarrow B,
\qquad c_2(a\wedge b)=[\widetilde a,\widetilde b].
\tag{2.3a}
\]
For \(ma=0\), choose \(\widetilde a\in\bar X\) over \(a\) and
\(f_ma\in\Lambda ^2\mathsf A\) with \(c_2(f_ma)=\widetilde a^m\).
Hartl--Mikhailov--Passi (HMP) define natural homomorphisms on the exterior torsion square
\[
\mathsf A\widehat *\mathsf A
:=
\operatorname{Tor}_1^{\mathbf Z}(\mathsf A,\mathsf A)
/
\langle\tau_m(a,a):ma=0\rangle
\tag{2.4}
\]
by
\[
\begin{aligned}
\delta _2\bigl(\overline{\tau_m(a,b)}\bigr)
 &=a\otimes\widetilde b^m-b\otimes\widetilde a^m
 &&\in \mathsf A\otimes B,\\
\delta _3\bigl(\overline{\tau_m(a,b)}\bigr)
 &=\binom m2\,s_3(a\otimes a\otimes b-a\otimes b\otimes b)
 &&\in\operatorname{SP}^3(\mathsf A).
\end{aligned}
\tag{2.5}
\]
These formulas are well defined. Set
\[
K_X:=\ker\delta _2\cap\ker\delta _3\subseteq \mathsf A\widehat *\mathsf A.
\tag{2.6}
\]
Put \(\Lambda=\mathbf Z[\bar X]\) and \(\Delta=\Delta(\bar X)\). The
degree-two polynomial homology group of HMP is
\[
P_2H_2(\bar X;\mathbf Z)
:=
H_2\!\left(
\mathbf Z\otimes_{\Lambda/\Delta^3}P_2B(\Lambda)
\right),
\tag{2.6a$'$}
\]
where \(\mathbf Z\) is the right \(\Lambda/\Delta^3\)-module induced by
augmentation. The natural projection from the ordinary bar construction
to the polynomial bar construction, after taking trivial coefficients,
induces the canonical map
\[
\rho^{\bar X}_{2*}:H_2(\bar X;\mathbf Z)
\longrightarrow P_2H_2(\bar X;\mathbf Z).
\tag{2.6a}
\]

HMP construct a natural homomorphism
\[
\widetilde j_X:\mathsf A\widehat *\mathsf A\longrightarrow H_2(\bar X;\mathbf Z).
\tag{2.6b}
\]
In the Hopf model this map is described as follows. Choose a free presentation
\[
1\longrightarrow R\longrightarrow F\xrightarrow{q}\bar X
\longrightarrow1.
\]
For \(\overline{\tau_m(a,b)}\), choose \(a_0,b_0\in F\) above
\(\widetilde a,\widetilde b\), and choose a commutator word \(w\in F'\)
representing \(f_ma\), so that \(q(w)=\widetilde a^m\). Since \(\bar X\)
has class at most two, \([b_0,w]\in R\cap F'\). Under the Hopf
isomorphism
\[
H_2(\bar X;\mathbf Z)\cong(R\cap F')/[F,R],
\]
HMP's map is characterized by
\[
\widetilde j_X\bigl(\overline{\tau_m(a,b)}\bigr)
=\overline{[b_0,w]}.
\tag{2.6b$'$}
\]
In the notation of HMP,
\(\widetilde j_X=\overline{\nu i}\,\delta_1\). Their construction proves
that (2.6b$'$) is independent of the choices, respects the exterior-torsion
relations, and is natural in \(X\).

Write \(j_X=\widetilde j_X|_{K_X}\). By HMP~\cite[Theorem~4.2]{HMP},
\[
\operatorname{im}j_X
=\ker\bigl(\rho^{\bar X}_{2*}:H_2(\bar X)\to P_2H_2(\bar X)\bigr).
\tag{2.6c}
\]
\end{theorem}

Combining Theorems~3.8 and~4.2 of \cite{HMP} with an identity in the
group ring gives the following surjectivity statement.

Let
\[
\kappa_X:H_2(\bar X;\mathbf Z)\longrightarrow\gamma _3(X)
\tag{2.6d}
\]
be the five-term transgression of the central extension
\[
1\longrightarrow\gamma _3(X)\longrightarrow X
\longrightarrow\bar X\longrightarrow1.
\tag{2.6e}
\]
We use the factor-set convention (1.23) for \(\kappa_X\). Under the Hopf
isomorphism used by HMP, it agrees with the transgression
\(\alpha _1\iota\nu^{-1}\) in (2.6k). Define
\[
\widetilde\Theta_X:\mathsf A\widehat *\mathsf A\longrightarrow\gamma _3(X),
\qquad
\widetilde\Theta_X(\xi)
=\bigl(\kappa_X\widetilde j_X(\xi)\bigr)^{-1},
\qquad
\Theta_X:=\widetilde\Theta_X|_{K_X}.
\tag{2.6f}
\]
The subgroup \(K_X\) and the maps \(\widetilde\Theta_X\), \(\Theta_X\)
are natural for homomorphisms of class-at-most-three groups.

\begin{corollary}[Surjectivity onto \(D_4(X)\), after Hartl--Mikhailov--Passi]
\label{corollary-2.3prime-theta-surjective}

The restriction \(\Theta_X\) is surjective:
\[
\Theta_X:K_X\twoheadrightarrow D_4(X).
\tag{2.6g}
\]
\end{corollary}

\begin{proof}
For a central subgroup \(C\le X\), write
\[
D_4(X,C)
:=
X\cap\left(1+I(C)\Delta(X)+\Delta(X)^4\right),
\tag{2.6g$'$}
\]
where \(I(C)\) is the ideal of \(\mathbf Z[X]\) generated by
\(\{c-1:c\in C\}\). Applying HMP~\cite[Theorem~3.8]{HMP} to (2.6e) gives
\[
D_4(X,\gamma _3X)\cap\gamma _3X
=\kappa_X(\ker\rho^{\bar X}_{2*}).
\tag{2.6h}
\]
Since \(\gamma _3(X)-1\subseteq\Delta(X)^3\), one has
\[
I(\gamma _3X)\Delta(X)+\Delta(X)^4=\Delta(X)^4
\tag{2.6i}
\]
and therefore \(D_4(X,\gamma _3X)=D_4(X)\). Since
\(D_4(X)\subseteq D_3(X)=\gamma _3(X)\), the left side of (2.6h) is
\(D_4(X)\). Equation (2.6c) now gives (2.6g).
\end{proof}

\begin{lemma}[Evaluation of the HMP map]
\label{lemma-2.4-hmp-map-evaluation}

Use the commutator convention \([x,y]=xyx^{-1}y^{-1}\). For every
symbol \(\overline{\tau_m(a,b)}\) in \(\mathsf A\widehat *\mathsf A\),
\[
\widetilde\Theta_X
\bigl(\overline{\tau_m(a,b)}\bigr)
=[\widehat b,\widehat a^m]^{-1}.
\tag{2.6j}
\]
Here \(\widehat a,\widehat b\in X\) are arbitrary lifts of
\(\widetilde a,\widetilde b\in\bar X\), and the right side is regarded
as an element of the central group \(\gamma _3(X)\).
\end{lemma}

\begin{proof}
Use the free presentation and the words \(a_0,b_0,w\) from (2.6b\('\)).
Since \(F\) is free, choose a homomorphism \(F\to X\) lifting \(q\). It
sends \(R\) into the central kernel, hence kills \([F,R]\) and factors
as a homomorphism
\[
\alpha _0:F/[F,R]\longrightarrow X.
\]
Its restriction to \(R/[F,R]\) has image in the central kernel
\(C=\gamma _3(X)\); denote that restriction by \(\alpha _1\). The proof of
HMP~\cite[Theorem~3.8]{HMP} identifies the five-term transgression with
\[
\kappa_X=\alpha _1\iota\nu^{-1},
\tag{2.6k}
\]
where \(\nu^{-1}\) is the Hopf isomorphism and \(\iota\) is the
inclusion of \((R\cap F')/[F,R]\) into \(R/[F,R]\). Formula (2.6b\('\))
therefore gives
\[
\kappa_X\widetilde j_X
\bigl(\overline{\tau_m(a,b)}\bigr)
=\alpha _1([b_0,w])
=[\alpha _0(b_0),\alpha _0(w)].
\tag{2.6l}
\]
Put \(\widehat a=\alpha _0(a_0)\) and \(\widehat b=\alpha _0(b_0)\).
Since \(q(w)=\widetilde a^m\),
\[
c:=\alpha _0(w)\widehat a^{-m}\in C.
\]
The subgroup \(C\) is central, so (2.6l) becomes
\[
[\widehat b,\alpha _0(w)]
=[\widehat b,c\widehat a^m]
=[\widehat b,\widehat a^m].
\]
The construction of \(\widetilde j_X\) makes the class of \([b_0,w]\) in
\((R\cap F')/[F,R]\) independent of the choices. Centrality of \(C\)
likewise makes the displayed commutator independent of the chosen lifts.
Equation (2.6f) gives (2.6j).
\end{proof}

For a finite abelian \(2\)-group
\[
\mathsf A=\bigoplus_{i=1}^sC_{2^{b_i}}e_i,\qquad b_1\le\cdots\le b_s,
\tag{2.7}
\]
the invariant-factor splitting induces the cyclic decomposition
\[
\mathsf A\widehat *\mathsf A
\cong
\bigoplus_{1\le i<j\le s}
C_{2^{b_i}}
\tau_{2^{b_i}}\!\left(e_i,2^{b_j-b_i}e_j\right).
\tag{2.8}
\]

\begin{theorem}[Tahara]
\label{theorem-2.5-tahara-family}

Specializing Tahara's counterexample family
\cite[pp.~392--393]{Tahara1977} to \(l=0\), for every \(k\ge2\) let
\(P_k\) be the finite \(2\)-group of nilpotency class three generated by
four elements
\(x_{11},x_{12},x_{13},x_{14}\), with
\[
\mathsf A_k=(P_k)_{\rm ab}
\cong
C_{2^k}e_1\oplus C_{2^{k+2}}e_2
\oplus C_{2^{k+4}}e_3\oplus C_{2^{k+4}}e_4.
\tag{2.9}
\]
In addition,
\[
\gamma_3(P_k)=\langle z=x_{21}^4\rangle\cong C_T,
\qquad T=2^{k+4},
\tag{2.10}
\]
and
\[
h_k:=x_{24}^{2^{k+5}}=z^{T/2}
\in D_4(P_k)\setminus\gamma_4(P_k).
\tag{2.11}
\]
In Tahara's notation~\cite{Tahara1977}, the six pair columns are indexed by
\(1\le i<j\le4\). A relation vector
\(u=(u_{ij})\in\mathbf Z^6\) is admissible when its mixed-tensor and
symmetric-cube components vanish in the corresponding quotient modules.

The second lower-central quotient is
\[
\gamma_2(P_k)/\gamma_3(P_k)
\cong
C_4f_1\oplus C_{2^k}f_2
\oplus C_{2^{k+2}}f_3\oplus C_{2^{k+4}}f_4.
\tag{2.12}
\]
For
\[
p_i=\overline{x_{1i}^{d(i)}}\in
\gamma_2(P_k)/\gamma_3(P_k),
\]
the \(f_3\)-coordinates are
\[
[f_3]p_1=2^{k-2},\qquad
[f_3]p_3=2^k,\qquad
[f_3]p_2=[f_3]p_4=0.
\tag{2.12a}
\]
For an admissible vector \(u\), Tahara's central map is
\[
\eta_H(u)
=
z^{\frac T4(u_{12}-u_{23})-\frac T2u_{13}}.
\tag{2.13}
\]
If \(d(i)=|e_i|\), the corresponding pair values are
\[
\rho_{ij}^{H_3}
=
[x_{1i}^{d(i)},x_{1j}]^{d(j)/d(i)},
\tag{2.13a}
\]
and \(\eta_H(u)=\prod_{i<j}(\rho_{ij}^{H_3})^{u_{ij}}\).

The vector
\[
(u_{12},u_{13},u_{23})=(2,1,2),
\qquad
u_{14}=u_{24}=u_{34}=0,
\tag{2.14}
\]
is admissible and has central value \(h_k\).
\end{theorem}

\subsection*{Tahara normalization data}

The invariant factors of \(\mathsf A_k\) and \(\gamma_2(P_k)/\gamma_3(P_k)\),
together with the group order, give \(|\gamma _3(P_k)|=2^{k+4}=T\), and
Tahara's relations give \(\gamma _3(P_k)=\langle x_{21}^4\rangle\).  Put
\[
f_j=x_{2j}\gamma _3(P_k)
\qquad(1\le j\le4).
\]
Tahara's power relations give, in \(\gamma _2(P_k)/\gamma _3(P_k)\),
\[
\begin{aligned}
p_1&=2^{k-2}f_3+2^{2k-1}f_4,\\
p_2&=2^{2k-2}f_2+2^{2k}f_4,\\
p_3&=-2^{2k-1}f_2+2^kf_3,
\end{aligned}
\tag{2.15}
\]
and \(p_4\) has zero \(f_3\)-coordinate, giving (2.12a). The same
relations give the pair values
\[
\begin{aligned}
\rho_{12}^{H_3}
 &=x_{21}^{2^{k+4}}=z^{T/4},\\
\rho_{13}^{H_3}
 &=x_{21}^{-2^{k+5}}=z^{-T/2},\\
\rho_{23}^{H_3}
 &=x_{21}^{-2^{k+4}}=z^{-T/4},
\end{aligned}
\tag{2.15a}
\]
while the other three pair values are trivial. Hence
\[
\prod_{i<j}(\rho_{ij}^{H_3})^{u_{ij}}
=z^{\frac T4(u_{12}-u_{23})-\frac T2u_{13}},
\tag{2.15b}
\]
which is (2.13). On the coordinate subspace \(u_{14}=u_{24}=u_{34}=0\),
Tahara's two admissibility conditions reduce to
\[
u_{12}=2u'_{12},\qquad
u_{23}=2u'_{23},\qquad
u'_{12}-u'_{23}\equiv0\pmod2,\qquad
u'_{12}-u_{13}\equiv0\pmod2.
\tag{2.15c}
\]
Taking \(u'_{12}=u'_{23}=u_{13}=1\) gives (2.14). Tahara's computation
on p.~393 identifies the resulting fourth dimension element, in agreement
with (2.11), as
\[
x_{24}^{2^{k+5}}=x_{21}^{2^{k+5}}
=(x_{21}^4)^{2^{k+3}}=z^{T/2}.
\tag{2.15d}
\]

\begin{lemma}[Compatibility of the HMP and Tahara maps]
\label{lemma-2.6-hmp-tahara-compatibility}

Put
\[
\bar P_k=P_k/\gamma _3(P_k),\qquad
B_k=\bar P_k'=\gamma _2(P_k)/\gamma _3(P_k),
\tag{2.16}
\]
and form \(K_{P_k}\) from the class-two group \(\bar P_k\) as in Theorem~\ref{theorem-2.3-hmp-map-degree-four}. For \(i<j\), set
\[
d_i=d(i),\qquad r_{ij}=d_j/d_i,
\qquad
g_{ij}^{(k)}
=\overline{\tau_{d_i}(e_i,r_{ij}e_j)}.
\tag{2.17}
\]
Then the decomposition (2.8) identifies an element
\(\xi=\sum_{i<j}u_{ij}g_{ij}^{(k)}\) with its six cyclic pair
coordinates, and defines a fixed coordinate isomorphism
\[
\operatorname{coord}_k:\mathsf A_k\widehat *\mathsf A_k
\xrightarrow{\ \cong\ }
\mathcal C_k:=\bigoplus_{i<j}C_{d_i}g_{ij}^{(k)}.
\tag{2.17a}
\]
With Tahara's mixed-tensor and symmetric-cube pair columns \(M_{ij}\) and
\(S_{ij}\), given in (2.22) and (2.24), define the Tahara admissible subgroup
\[
\mathcal L_k
:=
\left\{
u=(u_{ij})\in\mathcal C_k:
\sum_{i<j}u_{ij}M_{ij}=0,
\quad
\sum_{i<j}u_{ij}S_{ij}=0
\right\}.
\tag{2.17b}
\]
Then:

\begin{enumerate}
\def\labelenumi{\arabic{enumi}.}
\item
  \(\xi\in K_{P_k}\) if and only if \(u=(u_{ij})\) is an admissible Tahara
  relation vector.
\item
  Under these coordinates, the HMP map agrees with Tahara's central map:
  \[
  \Theta_{P_k}(\xi)=\eta_H(u)\in\gamma _3(P_k).
  \tag{2.18}
  \]
\item
  The maps \(\delta_2,\delta_3,\widetilde j,j,\kappa,\widetilde\Theta,\Theta\)
  and the subgroups \(K_X\) are natural. Hence the projection
  \(P=\prod_{k\ge2}P_k\to P_k\) carries \(K_P\) into \(K_{P_k}\) and
  commutes with \(\Theta\).
\end{enumerate}

Equivalently, the square
\[
\begin{CD}
K_{P_k}@>{\operatorname{coord}_k|_{K_{P_k}}}>{\cong}>\mathcal L_k\\
@V{\Theta_{P_k}}VV @VV{\eta_H}V\\
D_4(P_k)@>>>\gamma _3(P_k)
\end{CD}
\tag{2.19}
\]
commutes.
\end{lemma}

\begin{proof}
Fix \(i<j\), abbreviate \(d=d_i\), \(r=r_{ij}\), and choose in
\(\bar P_k\) the lifts \(x_{1i}\) of \(e_i\) and \(x_{1j}^r\) of
\(re_j\). By definition
\[
x_{1i}^{d}=p_i,\qquad (x_{1j}^{r})^d=x_{1j}^{d_j}=p_j
\quad\text{in }B_k.
\tag{2.20}
\]
Substitution into the first formula of (2.5) gives
\[
\delta _2(g_{ij}^{(k)})
=e_i\otimes p_j-r e_j\otimes p_i.
\tag{2.21}
\]
Equivalently,
\[
\delta _2(g_{ij}^{(k)})=-M_{ij},
\qquad
M_{ij}=r e_j\otimes p_i-e_i\otimes p_j.
\tag{2.22}
\]
The second formula of (2.5) gives
\[
\delta _3(g_{ij}^{(k)})
=r\binom d2(e_i\vee e_i\vee e_j)
-r^2\binom d2(e_i\vee e_j\vee e_j).
\tag{2.23}
\]
Tahara's symmetric-cube pair column is
\[
S_{ij}
=\binom{d_j}2(e_i\vee e_j\vee e_j)
-r\binom d2(e_i\vee e_i\vee e_j).
\tag{2.24}
\]
The two coefficients are related by the identity
\[
\binom{rd}2-r^2\binom d2
=d\,\frac{r(r-1)}2.
\tag{2.25}
\]
The element \(e_i\vee e_j\vee e_j\) has order dividing \(d_i=d\) in
\(\operatorname{SP}^3(\mathsf A_k)\), so (2.25) implies
\[
\delta _3(g_{ij}^{(k)})=-S_{ij}.
\tag{2.26}
\]
The six cyclic summands in (2.8) are independent, and therefore
\[
\xi\in\ker\delta _2\cap\ker\delta _3
\iff
\sum_{i<j}u_{ij}M_{ij}=0
\text{ and }
\sum_{i<j}u_{ij}S_{ij}=0,
\tag{2.27}
\]
which is Tahara's admissibility condition.

For the second assertion, formula
(2.6j), applied to the chosen lifts, gives
\[
\widetilde\Theta_{P_k}(g_{ij}^{(k)})
=[x_{1j}^{r},x_{1i}^{d}]^{-1}
=[x_{1i}^{d},x_{1j}]^{r}
=\rho_{ij}^{H_3}.
\tag{2.28}
\]
The middle equality uses the centrality of \(\gamma _3(P_k)\); the
inversion in (2.6f) matches Tahara's commutator convention. For
\(\xi=\sum_{i<j}u_{ij}g_{ij}^{(k)}\in K_{P_k}\), since \(\widetilde\Theta_{P_k}\)
is a homomorphism,
\[
\Theta_{P_k}(\xi)
=\widetilde\Theta_{P_k}(\xi)
=\prod_{i<j}
\left([x_{1j}^{r_{ij}},x_{1i}^{d_i}]^{-1}\right)^{u_{ij}}
=\prod_{i<j}(\rho_{ij}^{H_3})^{u_{ij}}
=\eta_H(u),
\tag{2.29}
\]
proving the commuting square.

Naturality of the HMP maps and of the five-term transgression gives (3).
\end{proof}

\section{A relative-bar description of the kernel}\label{relative-bar-kernel}

Define
\[
L_n^{\rm bar}
=
\left\{
(z_E,c_Q)\in
Z_2(E_n)\oplus\overline C_3(Q_n):
(\pi_n)_\#z_E+\partial c_Q\in D_n^{\rm bar}
\right\}
\tag{3.1}
\]
and
\[
R_n^{\rm bar}:L_n^{\rm bar}\longrightarrow D_n^{\rm bar},
\qquad
R_n^{\rm bar}(z_E,c_Q)
=
(\pi_n)_\#z_E+\partial c_Q.
\tag{3.2}
\]
The three-chain \(c_Q\) records the boundary ambiguity in passing from
normalized cycles to \(H_2(Q_n)\).

\begin{lemma}\label{lemma-3.1-relative-bar-image}

Under \(q_n:D_n^{\rm bar}\to F_n\),
\[
q_n(\operatorname{im}R_n^{\rm bar})
=
F_n\cap
\operatorname{im}[H_2(E_n)\to H_2(Q_n)],
\tag{3.3}
\]
and
\[
B_2(Q_n)\subseteq\operatorname{im}R_n^{\rm bar}.
\tag{3.4}
\]
\end{lemma}

\begin{proof}
For \((z_E,c_Q)\in L_n^{\rm bar}\), the class of
\(R_n^{\rm bar}(z_E,c_Q)\) in \(H_2(Q_n)\) is the image of \([z_E]\),
which proves one inclusion in (3.3). For the other, take \(x\) in the
right-hand side. Choose \(d\in D_n^{\rm bar}\) representing
\(x\), and choose \(z_E\in Z_2(E_n)\) whose homology class maps to
\(x\). Then
\[
d-(\pi_n)_\#z_E\in B_2(Q_n),
\]
so it equals \(\partial c_Q\) for some \(c_Q\in\overline C_3(Q_n)\).
Thus \(d=R_n^{\rm bar}(z_E,c_Q)\). Also,
\(\partial c_Q=R_n^{\rm bar}(0,c_Q)\) for every three-chain \(c_Q\).
\end{proof}

\begin{theorem}\label{theorem-3.2-relative-bar-realization}

For every group, prescribed \(N\)-series and degree \(n\),
\[
\operatorname{coker}R_n^{\rm bar}
\cong
\ker\phi_n.
\tag{3.5}
\]
\end{theorem}

\begin{proof}
By Lemma~\ref{lemma-3.1-relative-bar-image} and (3.4),
\[
\operatorname{coker}R_n^{\rm bar}
\cong
\frac{F_n}
{F_n\cap\operatorname{im}[H_2(E_n)\to H_2(Q_n)]}.
\tag{3.6}
\]
Exactness of (1.7) says
\[
\ker\kappa_n
=
\operatorname{im}[H_2(E_n)\to H_2(Q_n)].
\tag{3.7}
\]
The first isomorphism theorem gives
\[
\frac{F_n}{F_n\cap\ker\kappa_n}
\cong\kappa_n(F_n).
\tag{3.8}
\]
By Lemma~\ref{lemma-2.2-kernel-transgression}, the last group is \(\ker\phi_n\).
\end{proof}

\section{Finite separated criterion}\label{finite-separated-criterion}

Let \(G\) be finite, let
\[
G=H_1\ge H_2\ge\cdots\ge H_N=1,
\]
and suppose that the multiplication table of \(G\) and the subgroups \(H_i\)
are given explicitly.

\begin{theorem}\label{theorem-4.1-finite-criterion}
\[
A_\bullet\text{ realizes }H_\bullet
\iff
\operatorname{coker}R_n^{\rm bar}=0
\quad(1\le n<N).
\tag{4.1}
\]
The conditions on the right are effectively decidable by integer linear algebra.
\end{theorem}

\begin{proof}
By Theorem~\ref{theorem-3.2-relative-bar-realization}, realization is equivalent to the vanishing of
\(\operatorname{coker}R_n^{\rm bar}\) in every degree. Since \(H_N=1\),
only \(1\le n<N\) need be considered.

For a finite group, the normalized bar groups through degree three are free
abelian of finite rank. Their boundary maps are explicit integer matrices,
so Smith and Hermite normal forms compute the cycle and boundary lattices
and \(H_2(Q_n;\mathbf Z)\).

Choose a normalized section \(s_n:Q_n\to E_n\), compute its factor set
\(\alpha_n\) from the multiplication table, and choose representatives in
\(H_n\). Once the lattices \(A_n,A_{n+1}\subseteq\mathbf Z[G]\) have
been computed, formula (1.15) gives the integer matrix of
\(\vartheta_{n,s}\) directly on the normalized bar basis. Restricting this
matrix to \(Z_2(Q_n)\) and quotienting by \(B_2(Q_n)\) gives
\(\overline\vartheta_n\). Smith normal form then computes
\[
F_n=\ker\overline\vartheta_n.
\tag{4.1a}
\]
The same matrices compute \(D_n^{\rm bar}\), the pullback
\(L_n^{\rm bar}\), the map \(R_n^{\rm bar}\), and its cokernel.

After truncating weights larger than \(m\), the ideal \(A_m\) is generated
by products whose total weight first reaches \(m\) at the final factor;
every longer product belongs to the two-sided ideal generated by such a
prefix. Let \(S_m\) be the finite set of these prefixes. Since \(G\) is
finite, the two-sided ideal generated by \(S_m\) is the \(\mathbf Z\)-span
of
\[
\{gsh:g,h\in G,\ s\in S_m\}.
\tag{4.2}
\]
Writing these finitely many elements in the basis \(G\) of
\(\mathbf Z[G]\), Hermite normal form computes the lattice \(A_m\), completing
the finite algorithm.
\end{proof}

\section{Discrete and finite-level kernels}\label{discrete-finite-level-kernels}

Assume the hypotheses of Section~\ref{finite-quotient-saturation}.

\begin{lemma}\label{lemma-5.1-canonical-map}

The map
\[
\delta:\mathbf Z[G]\longrightarrow\varprojlim_j\mathbf Z[G_j]
\tag{5.1}
\]
is injective.
\end{lemma}

\begin{proof}
Let
\[
a=\sum_{\nu=1}^rc_\nu g_\nu
\tag{5.2}
\]
with the \(g_\nu\) distinct. Since the tower is separated and the set of
pairs \((g_\mu,g_\nu)\) is finite, some quotient \(G_j\) separates all
\(g_\nu\). If \(\delta(a)=0\), its image in \(\mathbf Z[G_j]\) is a
linear combination of distinct basis elements and hence every
\(c_\nu=0\).
\end{proof}

Define the finite-level symbol map
\[
\phi_n^\pi:H_n/H_{n+1}\longrightarrow
\prod_j A_{n,j}/A_{n+1,j},
\qquad
hH_{n+1}\longmapsto
\bigl((q_j(h)-1)+A_{n+1,j}\bigr)_j.
\tag{5.2a}
\]

\begin{lemma}\label{lemma-5.2-exact-kernel-formulas}

For every \(n\),
\[
\ker\phi_n
=
\frac{H_n\cap D_{n+1}^{\rm alg}}{H_{n+1}},
\qquad
\ker\phi_n^\pi
=
\frac{H_n\cap D_{n+1}^{\rm fin}}{H_{n+1}}.
\tag{5.3}
\]
\end{lemma}

\begin{proof}
The first equality is (1.4). The finite symbol of \(hH_{n+1}\) vanishes if
and only if \(q_j(h)-1\in A_{n+1,j}\) for every \(j\), equivalently if and
only if \(h\in D_{n+1}^{\rm fin}\).
\end{proof}

For the fixed tower, define the degree-\(n\) finite-quotient defect by
\[
\Omega_n
:=
\frac{\ker\phi_n^\pi}{\ker\phi_n}.
\tag{5.4}
\]
By (5.3) and the third isomorphism theorem,
\[
\Omega_n
\cong
\frac{H_n\cap D_{n+1}^{\rm fin}}
     {H_n\cap D_{n+1}^{\rm alg}}.
\tag{5.5}
\]

\section{Bounded width and descent}\label{bounded-width-descent}

\subsection{Losey expressions}\label{losey-expressions}

A degree-\((n+1)\) Losey expression for \(a\in A_{n+1}\) is an equality
\[
a=
\sum_{u=1}^sc_u
\prod_{v=1}^{r_u}(x_{u,v}-1),
\tag{6.1}
\]
where
\[
c_u\in\mathbf Z\setminus\{0\},\qquad
x_{u,v}\in H_{w_{u,v}},\qquad
\sum_vw_{u,v}\ge n+1.
\tag{6.2}
\]
We use reduced expressions, with zero-coefficient summands omitted; the
empty sum represents \(0\).

Weights larger than \(n+1\) may be replaced by \(n+1\). Define
\[
\|\mathcal W\|
:=
\sum_{u=1}^s|c_u|(1+r_u).
\tag{6.3}
\]
For \(h\in H_n\cap D_{n+1}^{\rm alg}\), let
\[
\mu_n(G,H_\bullet;h)
:=
\min\{\|\mathcal W\|:\mathcal W\text{ expresses }h-1\}.
\tag{6.4}
\]
For a compatible \(h=(h_j)\in H_n\cap D_{n+1}^{\rm fin}\), write
\[
\mu_{n,j}(h)
:=
\mu_n(G_j,H_{\bullet,j};h_j).
\tag{6.5}
\]
Projection to a quotient cannot increase this complexity; hence
\[
\mu_{n,1}(h)\le\mu_{n,2}(h)\le\cdots.
\tag{6.6}
\]

\begin{theorem}\label{theorem-6.1-discrete-descent-bounded-width}

For every compatible \(h\in H_n\cap D_{n+1}^{\rm fin}\),
\[
h\in D_{n+1}^{\rm alg}
\iff
\sup_j\mu_{n,j}(h)<\infty.
\tag{6.7}
\]
\end{theorem}

\begin{proof}
A Losey expression for \(h-1\) in \(\mathbf Z[G]\) projects to an
expression of the same complexity at every finite level.

For the converse, suppose all levels admit an expression of complexity at most
\(\beta\). Pad expressions by retaining factors that project to \(1\), so that
coefficients, factor slots and truncated weight labels are preserved. Let
\(\mathscr W_j(\beta)\) be the finite nonempty set of padded expressions of
complexity at most \(\beta\) representing \(h_j-1\). Factorwise projection maps
\(\mathscr W_{j+1}(\beta)\) to \(\mathscr W_j(\beta)\).

These sets form an infinite finitely branching rooted tree. K\"onig's lemma
gives a compatible branch with fixed coefficients, slots and weight labels.
Each slot defines an element of \(G=\varprojlim_jG_j\); a slot of weight \(w\)
lies in \(H_w=\varprojlim_jH_{w,j}\). The branch therefore defines a Losey
expression \(\mathcal W\in\mathbf Z[G]\), after deleting zero summands.

Its image is \(h_j-1\) at every finite level. By Lemma~\ref{lemma-5.1-canonical-map},
\(\mathcal W=h-1\) in \(\mathbf Z[G]\), so
\(h\in D_{n+1}^{\rm alg}\).
\end{proof}

\begin{corollary}\label{corollary-6.2-exact-pro-integral-criterion}
\[
\Omega_n=0
\iff
\text{every }h\in H_n\cap D_{n+1}^{\rm fin}
\text{ has bounded finite-level Losey width}.
\tag{6.8}
\]

\end{corollary}

\begin{proof}
By (5.5), \(\Omega_n=0\) if and only if
\(H_n\cap D_{n+1}^{\rm fin}\subseteq D_{n+1}^{\rm alg}\).  Apply
Theorem~\ref{theorem-6.1-discrete-descent-bounded-width}.
\end{proof}

\section{The Tahara map and a parity constraint}\label{tahara-parity}

The parity obstruction in Theorem~\ref{theorem-8.3-failure-of-discrete-descent} rests on the following congruence.

\begin{lemma}\label{lemma-7.1-six-pair-congruence}

Every admissible Tahara relation vector satisfies
\[
u_{12}-u_{23}\equiv0\pmod4.
\tag{7.1}
\]
\end{lemma}

\begin{proof}
Inspect the \(e_2\otimes f_3\)-coordinate in the mixed-tensor columns
(2.22); by (2.9) and (2.12) this summand of \(\mathsf A_k\otimes B_k\) has order
\(2^{k+2}\).  Only pair columns \(12\) and \(23\) contribute. By (2.12a), their
coefficients are respectively
\[
\frac{2^{k+2}}{2^k}\,2^{k-2}=2^k,
\qquad
-2^k.
\]
Admissibility therefore gives
\[
2^k(u_{12}-u_{23})=0\quad\text{in }C_{2^{k+2}},
\]
which is equivalent to (7.1).
\end{proof}

\begin{lemma}\label{lemma-7.2-central-value}

For every admissible Tahara relation vector \(u\),
\[
\eta_H(u)=z^{\frac T2u_{13}}.
\tag{7.2}
\]
\end{lemma}

\begin{proof}
By Lemma~\ref{lemma-7.1-six-pair-congruence}, write \(u_{12}-u_{23}=4a\). Substitute into (2.13):
\[
\eta_H(u)
=z^{Ta-\frac T2u_{13}}
=z^{-\frac T2u_{13}}
=z^{\frac T2u_{13}}.
\tag{7.3}
\]
The last equality uses \(z^{-T/2}=z^{T/2}\) in \(C_T\).
\end{proof}

\begin{corollary}[Odd \((1,3)\)-coordinate]
\label{corollary-7.3-odd-coordinate}

Every
\[
\xi_k\in K_{P_k}
\quad\text{with}\quad
\Theta_{P_k}(\xi_k)=h_k
\tag{7.4}
\]
has odd coefficient on \(g_{13}^{(k)}\).
\end{corollary}

\begin{proof}
By Lemma~\ref{lemma-2.6-hmp-tahara-compatibility} the coefficient is \(u_{13}\). Equations (2.11) and (7.2)
give
\[
z^{\frac T2u_{13}}=z^{T/2}
\quad\text{in }C_T,
\tag{7.5}
\]
which is equivalent to \(u_{13}\equiv1\pmod2\).
\end{proof}

\begin{lemma}\label{lemma-7.4-product-lower-central}

For
\[
P=\prod_{k\ge2}P_k,\qquad h=(h_k),
\tag{7.6}
\]
one has
\[
\gamma_s(P)=\prod_{k\ge2}\gamma_s(P_k)
\qquad(1\le s\le3),
\tag{7.7}
\]
and in particular \(h\in\gamma_3(P)\).
\end{lemma}

\begin{proof}
Functoriality gives
\(\gamma_s(P)\subseteq\prod_k\gamma_s(P_k)\). For \(1\le i\le4\), let \(x_i\in P\) be the element with coordinates
\(x_{i,k}=x_{1i}\in P_k\).

Let \(y=(y_k)\in\prod_k\gamma_3(P_k)\). Since each \(P_k\) has class
three, \(\gamma_3(P_k)\) is central and is generated by the fixed finite
family of basic triple commutators
\([x_{i,k},x_{j,k},x_{\ell,k}]\). Write the corresponding exponents of
\(y_k\) as \(e_{ij\ell,k}\), and set
\[
a_{ij\ell}=(x_{i,k}^{e_{ij\ell,k}})_k\in P.
\]
In every group of class at most three and for every \(m\in\mathbf Z\),
the collection formulas, together with the centrality of
\(\gamma_3\), give
\[
[u^m,v,w]=[u,v,w]^m,
\qquad
[u^m,v]\equiv[u,v]^m\pmod{\gamma_3}.
\]
The first identity gives
\[
[a_{ij\ell},x_j,x_\ell]_k
=
[x_{i,k},x_{j,k},x_{\ell,k}]^{e_{ij\ell,k}}.
\tag{7.7a}
\]
Multiplying over the finitely many basic triples yields
\(y\in\gamma_3(P)\), and hence
\(\gamma_3(P)=\prod_k\gamma_3(P_k)\).

Now let \(y\in\prod_k\gamma_2(P_k)\). Modulo
\(\prod_k\gamma_3(P_k)=\gamma_3(P)\), each coordinate is a product of the
fixed basic weight-two commutators. If \(e_{ij,k}\) are the corresponding
exponents, the second identity shows that the commutators
\[
[(x_{i,k}^{e_{ij,k}})_k,x_j]
\]
realize these coordinates modulo \(\gamma_3(P)\). Their product therefore
differs from \(y\) by an element of \(\gamma_3(P)\), already treated above.
Therefore \(\gamma_2(P)=\prod_k\gamma_2(P_k)\).
\end{proof}

\begin{lemma}\label{lemma-7.5-augmentation-powers}

For the lower-central \(N\)-series \(H_s=\gamma_s(X)\),
\[
A_m(X,\gamma_\bullet)=\Delta(X)^m.
\tag{7.8}
\]
\end{lemma}

\begin{proof}
The identity
\[
[x,y]-1=\bigl((x-1)(y-1)-(y-1)(x-1)\bigr)x^{-1}y^{-1}
\]
and induction on \(s\) give \(x-1\in\Delta(X)^s\) for
\(x\in\gamma_s(X)\). Hence
\(A_m\subseteq\Delta^m\). The reverse inclusion follows because
\(\Delta^m\) is generated by products of \(m\) weight-one terms
\((g_i-1)\).
\end{proof}

\begin{corollary}[\(D_4\) via the relative-bar cokernel]
\label{corollary-7.6-D4-via-bar-cokernel}

Let \(X\) be a group with \(\gamma_4(X)=1\). For the lower-central
\(N\)-series \(H_s=\gamma_s(X)\),
\[
D_4(X)\cong\operatorname{coker}R_3^{\rm bar}.
\tag{7.9}
\]
\end{corollary}

\begin{proof}
By Lemma~\ref{lemma-7.5-augmentation-powers}, \(A_m(X,\gamma_\bullet)=\Delta(X)^m\), so
\(D_m^{\rm alg}(X,\gamma_\bullet)=X\cap(1+\Delta(X)^m)=D_m(X)\) for
every \(m\), matching (1.4\('\)). By (1.4),
\[
\ker\phi_3=\frac{\gamma_3(X)\cap D_4(X)}{\gamma_4(X)}
=\gamma_3(X)\cap D_4(X).
\]
Since \(D_4(X)\subseteq D_3(X)=\gamma_3(X)\) for every group, the right
side is \(D_4(X)\). Theorem~\ref{theorem-3.2-relative-bar-realization} gives
\(\operatorname{coker}R_3^{\rm bar}\cong\ker\phi_3\), proving (7.9).
\end{proof}

\section{The exterior-torsion obstruction}\label{exterior-torsion-obstruction}

\subsection{Bounded-exponent image}\label{bounded-exponent-image}

For an abelian \(2\)-group \(\mathsf A\), put
\[
\mathcal T_M(\mathsf A)
:=
\operatorname{im}
\left(
\mathsf A[2^M]\widehat *\mathsf A[2^M]\to \mathsf A\widehat *\mathsf A
\right).
\tag{8.1}
\]

\begin{lemma}\label{lemma-8.1-pair-coordinate-divisibility}

Let
\[
\mathsf A=\bigoplus_{r=1}^sC_{2^{b_r}}e_r,
\qquad b_1\le\cdots\le b_s.
\tag{8.2}
\]
The \((i,j)\)-coordinate of \(\mathcal T_M(\mathsf A)\) in (2.8) is contained in
\[
2^{\max(b_i-M,0)}C_{2^{b_i}}.
\tag{8.3}
\]
\end{lemma}

\begin{proof}
Put \(c_r=\min(b_r,M)\). Then
\[
\mathsf A[2^M]=\bigoplus_rC_{2^{c_r}}e'_r,
\qquad
e'_r=2^{b_r-c_r}e_r.
\tag{8.4}
\]
For \(i<j\), the cross-pair generator is
\[
\tau_{2^{c_i}}
\left(e'_i,2^{c_j-c_i}e'_j\right).
\tag{8.5}
\]
Use the free resolution
\[
0\longrightarrow\mathbf Z
\xrightarrow{\,2^a\,}\mathbf Z
\longrightarrow C_{2^a}\longrightarrow0.
\tag{8.6}
\]
The chain map induced by \(C_{2^{c_i}}\to C_{2^{b_i}}\),
\(1\mapsto2^{b_i-c_i}\), has degree-one component \(1\) and degree-zero
component \(2^{b_i-c_i}\); tensoring with the second cyclic map and
taking the Tor kernel sends (8.5) to
\[
2^{b_i-c_i}
\tau_{2^{b_i}}
\left(e_i,2^{b_j-b_i}e_j\right).
\tag{8.7}
\]
Since \(b_i-c_i=\max(b_i-M,0)\), this proves the asserted divisibility.
Quotienting by diagonal symbols cannot mix distinct unordered pair
coordinates.
\end{proof}

For \(\mathsf A_k\) and pair \((1,3)\),
\[
\operatorname{pr}_{13}\mathcal T_M(\mathsf A_k)
\subseteq
2^{\max(k-M,0)}C_{2^k}g_{13}^{(k)}.
\tag{8.8}
\]
Thus the \((1,3)\)-coordinate is even when \(k>M\).

\begin{lemma}\label{lemma-8.2-uniform-exponent-bound}

For every abelian group \(\mathsf A\) and \(\xi\in \mathsf A\widehat *\mathsf A\), there exists
\(M\) such that, for every finite abelian \(2\)-group quotient \(Q\) of
\(\mathsf A\), the image of \(\xi\) lies in \(\mathcal T_M(Q)\).
\end{lemma}

\begin{proof}
Let \(t\mathsf A\) be the torsion subgroup of \(\mathsf A\).  Since \(\mathbf Z\) is a
principal ideal domain, \(\mathsf A/t\mathsf A\) is flat and
\(\operatorname{Tor}^{\mathbf Z}_i=0\) for \(i\ge2\).  Applying the long
exact Tor sequence to
\[
0\longrightarrow t\mathsf A\longrightarrow \mathsf A\longrightarrow \mathsf A/t\mathsf A
\longrightarrow0
\]
first in the left variable with second argument \(t\mathsf A\), and then in the
right variable with first argument \(\mathsf A\), gives natural isomorphisms
\[
\operatorname{Tor}^{\mathbf Z}_1(t\mathsf A,t\mathsf A)
\xrightarrow{\ \cong\ }
\operatorname{Tor}^{\mathbf Z}_1(\mathsf A,t\mathsf A)
\xrightarrow{\ \cong\ }
\operatorname{Tor}^{\mathbf Z}_1(\mathsf A,\mathsf A).
\tag{8.8a}
\]
Choose \(\widetilde\xi\in\operatorname{Tor}^{\mathbf Z}_1(\mathsf A,\mathsf A)\) mapping
to \(\xi\), and let \(\eta\in\operatorname{Tor}^{\mathbf Z}_1(t\mathsf A,t\mathsf A)\)
be its preimage under (8.8a). The Eilenberg--Mac Lane presentation of
\(\operatorname{Tor}^{\mathbf Z}_1(t\mathsf A,t\mathsf A)\) is generated by symbols
\[
\tau_m(x,y),\qquad mx=my=0.
\tag{8.9}
\]
Represent \(\eta\) by a finite sum of such symbols. If \(\eta=0\), take
\(M=0\); otherwise let \(M\) be the largest \(2\)-adic valuation of the
moduli in this expression.
Write a modulus as \(m=2^vu\) with \(u\) odd.  Multiplication by \(u\) is
an automorphism of every finite abelian \(2\)-group \(Q\); hence
\(m\bar x=m\bar y=0\) implies
\(\bar x,\bar y\in Q[2^v]\subseteq Q[2^M]\). Naturality of
\(\operatorname{Tor}_1^{\mathbf Z}\) therefore factors the projected
symbol through
\[
\operatorname{Tor}_1(Q[2^M],Q[2^M])
\longrightarrow\operatorname{Tor}_1(Q,Q).
\tag{8.9a}
\]
Passing to the exterior quotient gives the same factorization in \(Q\widehat *Q\).
\end{proof}

\begin{theorem}[Failure of discrete descent]
\label{theorem-8.3-failure-of-discrete-descent}

Let
\[
P=\prod_{k\ge2}P_k,\qquad h=(h_k).
\tag{8.10}
\]
Equip \(P\) with its lower central series and with the tower of finite
subproducts
\[
P\longrightarrow P^{(j)}:=\prod_{2\le k\le j}P_k,
\tag{8.10a}
\]
and let \(D_4^{\rm fin}(P)\) be the finite-quotient saturation for this
tower.  Then
\[
h\in\gamma_3(P)\cap D_4^{\rm fin}(P)
\setminus D_4(P).
\tag{8.11}
\]
\end{theorem}

\begin{proof}
The tower (8.10a) is cofinal among the continuous finite quotients of
\(P\): the open kernel of a homomorphism from \(P\) to a finite discrete
group contains \(\prod_{k>j}P_k\) for some \(j\), so the homomorphism
factors through \(P^{(j)}\).  By
Lemmas~\ref{lemma-7.4-product-lower-central}
and~\ref{lemma-7.5-augmentation-powers} the lower central series of \(P\)
is closed and coordinatewise, and \(A_4(P,\gamma_\bullet)=\Delta(P)^4\).

Lemma~\ref{lemma-7.4-product-lower-central} gives \(h\in\gamma_3(P)\).
Each \(h_k\) lies in \(D_4(P_k)\), and the inclusion \(P_k\to P^{(j)}\)
carries \(D_4(P_k)\) into \(D_4(P^{(j)})\).  Hence the image
\((h_2,\ldots,h_j)\) of \(h\) lies in \(D_4(P^{(j)})\) for every \(j\), that
is, \(h\in D_4^{\rm fin}(P)\).

Assume \(h\in D_4(P)\). The product has nilpotency class at most three,
so \(\gamma_4(P)=1\). Corollary~\ref{corollary-2.3prime-theta-surjective} supplies
\[
\xi\in K_P\subseteq P_{\rm ab}\widehat *P_{\rm ab}
\quad\text{with}\quad
\Theta_P(\xi)=h.
\tag{8.12}
\]
Choose \(M\) as in Lemma~\ref{lemma-8.2-uniform-exponent-bound} and project to \(P_k\). Then
\[
\xi_k\in K_{P_k}\cap\mathcal T_M(\mathsf A_k),
\qquad
\Theta_{P_k}(\xi_k)=h_k.
\tag{8.13}
\]
Choose \(k>M\). Lemma~\ref{lemma-8.1-pair-coordinate-divisibility} says that the \((1,3)\)-coordinate of
\(\xi_k\) is even, while Corollary~\ref{corollary-7.3-odd-coordinate} says it is odd. This
contradiction proves \(h\notin D_4(P)\).
\end{proof}

\begin{corollary}\label{corollary-8.4-finite-subproduct-width-divergence}
The finite-level Losey widths of \(h\) are unbounded: with
\(h^{(j)}=(h_2,\ldots,h_j)\in P^{(j)}\),
\[
\sup_{j\ge2}
\mu_3\!\left(P^{(j)},\gamma_\bullet;h^{(j)}\right)=\infty.
\tag{8.14}
\]
\end{corollary}

\begin{proof}
If the widths were bounded, Theorem~\ref{theorem-6.1-discrete-descent-bounded-width} would produce one discrete
expression for \(h-1\), putting \(h\) in \(D_4(P)\), contrary to Theorem~\ref{theorem-8.3-failure-of-discrete-descent}.
\end{proof}

For the tower (8.10a), formula (5.5) and
Theorem~\ref{theorem-8.3-failure-of-discrete-descent} give
\[
\Omega_3
\cong
\frac{\gamma_3(P)\cap D_4^{\rm fin}(P)}
     {\gamma_3(P)\cap D_4(P)}
\ne0.
\tag{8.15}
\]

\end{document}